\documentclass[a4paper]{amsart}
\usepackage{graphicx}
\usepackage{amssymb}
\usepackage{amsmath}
\usepackage{amsthm}
\usepackage{amscd}
\usepackage[all,2cell]{xy}
\UseAllTwocells \SilentMatrices
\newtheorem{thm}{Theorem}[section]

\newtheorem{lem}[thm]{Lemma}

\newtheorem{prop}[thm]{Proposition}
\theoremstyle{definition}
\newtheorem{defn}[thm]{Definition}
\theoremstyle{remark}
\newtheorem{rem}[thm]{\bf Remark}
\numberwithin{equation}{section}

\begin{document}
\title[The derived-discrete algebras and standard equivalences]{The derived-discrete algebras and standard equivalences}
\author[Xiao-Wu Chen, Chao Zhang] {Xiao-Wu Chen, Chao Zhang$^*$}

\thanks{$^*$ The corresponding author}
%\thanks{}
\date{\today}

\thanks{E-mail:
xwchen$\symbol{64}$mail.ustc.edu.cn, czhang2$\symbol{64}$gzu.edu.cn}
%\thanks{this is NOT for distribution}

\maketitle

\dedicatory{}%
\commby{}%
%\begin{center}
%\end{center}

\begin{abstract}
We prove that any derived equivalence between derived-discrete algebras of finite global dimension is standard, that is,  isomorphic to the derived tensor functor by a two-sided tilting complex.
\end{abstract}

\section{Introduction}

Let $k$ be an algebraically closed field. For a finite dimensional $k$-algebra $A$, we denote by $A\mbox{-mod}$ the category of finite dimensional $A$-modules and by $\mathbf{D}^b(A\mbox{-mod})$ the bounded derived category. By a derived equivalence between two algebras $A$ and $B$, we mean a $k$-linear triangle equivalence between $\mathbf{D}^b(A\mbox{-mod})$ and $\mathbf{D}^b(B\mbox{-mod})$.

It is an open question in \cite{Ric} whether any derived equivalence between two algebras is standard, that is, isomorphic to the derived tensor functor by a two-sided tilting complex. The question is answered affirmatively for hereditary algebras \cite{MY}, (anti-)Fano algebras \cite{Mina}, triangular algebras \cite{Chen} and the algebra of dual numbers \cite{CY}.

The main result of this paper answers the open question affirmatively for derived-discrete algebras of finite global dimension; see Theorem \ref{thm:B}. Recall from \cite{CY} that the open question is affirmative for an algebra $A$,  provided that the category $A\mbox{-proj}$ of finite dimensional projective $A$-modules is $\mathbf{K}$-standard in the sense of \cite{CY}. Then the main result boils down to the following result: the category $A(r, N)\mbox{-proj}$  is $\mathbf{K}$-standard, where $A(r, N)$ is a certain Nakayama algebra given by a cyclic quiver with $N$ vertices and $r$ consecutive quadratic zero relations; see Theorem \ref{thm:A}.

Derived-discrete algebras are introduced in \cite{Vos}. They form an important class of algebras for testing conjectures and carrying out concrete calculation, since their bounded derived categories are very much accessible.  We mention that we rely on the classification \cite{BGS} of derived-discrete algebras up to derived equivalence.

We draw an immediate consequence of Theorem \ref{thm:B}: the group of standard derived autoequivalences for a derived-discrete algebra of finite global dimension, as calculated in \cite{KYan, BPP}, coincides with the whole group of derived autoequivalences.

The paper is structured as follows. In Section 2, we recall basic facts on triangle functors and determine the centers of a certain  triangulated category. In Section 3, we formulate Theorem \ref{thm:A}, from which we deduce Theorem \ref{thm:B}. In Section 4, we prove Theorem \ref{thm:A} for the case $r=1$. In Section 5, we sketch the proof for the case $r>1$, since it is very close to the one in Section 4. We mention that the argument here is similar to, but more complicated than, the one in \cite[Section 7]{CY}.

\section{Triangle functors and centers}

In this section, we first recall from \cite{CY} some  facts on triangle functors and the centers of a triangulated category. We determine the centers of a certain triangulated category. For triangulated categories, we refer to \cite{Hap88, Zim}.

Let $\mathcal{T}$ and $\mathcal{T}'$ be  triangulated categories, whose suspension functors are denoted by $\Sigma$ and $\Sigma'$, respectively. A triangle $(F, \omega)\colon \mathcal{T}\rightarrow \mathcal{T}'$ consists of an additive functor $F$ with a natural isomorphism $\omega\colon F\Sigma\rightarrow \Sigma' F$,  which sends an exact triangle to an exact triangle. More precisely, an exact triangle $X\stackrel{f}\rightarrow Y\stackrel{g}\rightarrow Z \stackrel{h}\rightarrow \Sigma(X)$ in $\mathcal{T}$ is sent to an exact triangle $F(X)\stackrel{F(f)}\rightarrow F(Y)\stackrel{F(g)}\rightarrow F(Z) \xrightarrow{\omega_X\circ F(h)} \Sigma'(FX)$ in $\mathcal{T}'$. We call $\omega$ the \emph{connecting isomorphism} for $F$. We sometimes suppress $\omega$ and say that $F$ is a triangle functor. For example, the identity functor $({\rm Id}_\mathcal{T}, {\rm Id}_\Sigma)\colon \mathcal{T}\rightarrow \mathcal{T}$ is a triangle functor, which is often abbreviated as ${\rm Id}_\mathcal{T}$. A natural transformation $\eta\colon (F, \omega)\rightarrow (F_1, \omega_1)$ between triangle functors is a natural transformation respecting the connecting isomorphisms, that is, $\omega_1\circ \eta \Sigma=\Sigma'\eta\circ \omega$.

The following well-known fact is given in \cite[Lemma 2.3]{CY}.

\begin{lem}\label{lem:adjust}
Let $(F, \omega)\colon \mathcal{T}\rightarrow \mathcal{T}'$ be a triangle functor. Suppose that we are given an isomorphism $\delta_X\colon F(X)\rightarrow F_1(X)$ in $\mathcal{T}'$ for each object $X\in \mathcal{T}$. Then there is a unique triangle functor $(F_1, \omega_1)$ such that $\delta\colon (F, \omega)\rightarrow (F_1, \omega_1)$ is an isomorphism between triangle functors.
\end{lem}

We will call the given isomorphisms $\delta_X$'s the \emph{adjusting isomorphisms}, and say that the new triangle functor $(F_1, \omega_1)$ is \emph{adjusted} from $(F, \omega)$.

\begin{proof}
The functor $F_1$ necessarily acts on a morphism $f\colon X\rightarrow Y$ by $\delta_Y\circ F(f)\circ (\delta_X)^{-1}$, and the connecting isomorphism $\omega_1$ is given by $(\Sigma'\delta)\circ \omega\circ (\delta\Sigma)^{-1}$.
\end{proof}

In what follows, we fix a field $k$. The triangulated category $\mathcal{T}$ is  required to be $k$-linear, which is Hom-finite and Krull-Schmidt. All the functors will be assumed to be $k$-linear.

For a set $\mathcal{M}$ of morphisms in $\mathcal{T}$, we denote by ${\rm obj}(\mathcal{M})$ the full subcategory formed by those objects, which are either the domain or the codomain of a morphism in $\mathcal{M}$. We say that  $\mathcal{M}$  $k$-linearly \emph{spans} $\mathcal{T}$,  provided that each morphism in ${\rm obj}(\mathcal{M})$ is a $k$-linear combination of the identity morphisms and  composition of morphisms from $\mathcal{M}$, and that each object in $\mathcal{T}$ is isomorphic to a finite direct sum of objects in ${\rm obj}(\mathcal{M})$.

The following fact follows from \cite[Lemmas 2.5 and 2.3]{CY}.

\begin{lem}\label{lem:span}
Let $\mathcal{M}$ be a spanning set of morphisms in $\mathcal{T}$. Assume that $(F, \omega)\colon \mathcal{T}\rightarrow \mathcal{T}$ is a $k$-linear triangle endofunctor such that $F(f)=f$ for each $f\in \mathcal{M}$. Then there is a unique natural isomorphism $\omega'\colon \Sigma\rightarrow \Sigma$ and  a unique natural isomorphism $\theta\colon (F, \omega)\rightarrow ({\rm Id}_\mathcal{T}, \omega')$ between triangle functors satisfying $\theta_S={\rm Id}_S$ for each object $S$ from ${\rm obj}(\mathcal{M})$. \hfill $\square$
\end{lem}

We denote by $Z(\mathcal{T})$ the \emph{center} of $\mathcal{T}$, which consists of all natural transformations $\lambda\colon {\rm Id}_\mathcal{T}\rightarrow {\rm Id}_\mathcal{T}$. Then $Z(\mathcal{T})$ is a commutative $k$-algebra, whose multiplication is induced by the composition of natural transformations. We denote by $Z_\vartriangle(\mathcal{T})$ the \emph{triangle center}, which is a subalgebra of $Z(\mathcal{T})$ consisting of natural transformations $\lambda\colon {\rm Id}_\mathcal{T}\rightarrow {\rm Id}_\mathcal{T}$ between triangle functors. Then $\lambda\in Z(\mathcal{T})$ belongs to $Z_\vartriangle(\mathcal{T})$ if and only if $\Sigma(\lambda_X)=\lambda_{\Sigma(X)}$ for any object $X\in \mathcal{T}$.

The center is related to almost-vanishing endomorphisms; see \cite{Ku}. Let $X$ be an indecomposable object in $\mathcal{T}$. An nonzero non-invertible endomorphism $\Delta\colon X\rightarrow X$ is \emph{almost-vanishing} provided that $\Delta\circ f=0=g\circ \Delta $ for any non-retraction morphism $f\colon Y\rightarrow X$ and non-section $g\colon X\rightarrow Z$. We observe that $\Delta$ is almost-vanishing if and only if  it fits into an almost split triangle $\Sigma^{-1}(X)\rightarrow E\rightarrow X\stackrel{\Delta}\rightarrow X$; see \cite[I.4.1]{Hap88}. Here, we identity $X$ with $\Sigma (\Sigma^{-1}X)$.  Moreover, we have $\Delta^2=0$.

We denote by ${\rm ind} \mathcal{T}$ a complete set of representatives  of isoclasses of indecomposable objects in $\mathcal{T}$. We observe that a natural transformation $\eta\colon F\rightarrow F_1$ between endofunctors on $\mathcal{T}$ is uniquely determined by its restriction on ${\rm ind}\mathcal{T}$.

The following construction of central elements is due to \cite[Lemma 2.2]{Ku}; see also \cite[Lemma 2.7]{CY}.

\begin{lem}\label{lem:center}
Assume that $X\in {\rm ind}\mathcal{T}$ has  an almost-vanishing endomorphism $\Delta$. Then there is a unique element $\delta(X)\in Z(\mathcal{T})$ satisfying $\delta(X)_X=\Delta$ and $\delta(X)_Y=0$ for any $Y\in {\rm ind}\mathcal{T}$, different from $X$. \hfill $\square$
\end{lem}

The following fact, due to \cite[Proposition 2.6]{CY},  will be used often.

\begin{lem}\label{lem:scalar}
 Assume that $X\stackrel{f}\rightarrow Y \stackrel{g}\rightarrow Z \stackrel{h}\rightarrow \Sigma(X)$ is an exact triangle in $\mathcal{T}$ with $g\neq 0$ and $h\neq 0$ such that ${\rm End}_\mathcal{T}(Z)$ either equals $k{\rm Id}_Z$ or $k{\rm Id}_Z\oplus k\Delta$, where the endomorphism $\Delta$ on $Z$ is almost-vanishing. Then for a nonzero scalar $\lambda\in k$, the triangle $X\stackrel{f}\rightarrow Y \stackrel{g}\rightarrow Z \xrightarrow{\lambda h} \Sigma(X)$ is exact if and only if $\lambda =1$. \hfill $\square$
\end{lem}

We will consider the following condition on $\mathcal{T}$.

\hskip 5pt

\noindent (A1)\quad There is a disjoint union ${\rm ind}\mathcal{T}=S\cup S'$ such that ${\rm End}_\mathcal{T}(X)=k{\rm Id}_X\oplus k \Delta_X$ for each $X\in S$ with $\Delta_X$ a fixed almost-vanishing endomorphism on $X$, and that ${\rm End}_\mathcal{T}(Y)=k{\rm Id}_Y$ for each $Y\in S'$.

\hskip 5pt

The triangulated category $\mathcal{T}$ is a \emph{block}, provided that $\mathcal{T}$ does not admit a decomposition into the product of  two nonzero triangulated subcategories.  It is \emph{non-degenerate} if there is a nonzero non-invertible morphism $X\rightarrow Y$ for some  $X, Y\in {\rm ind}\mathcal{T}$; for details, see \cite[Section 4]{CR}.

\begin{prop}\label{prop:A1}
Let $\mathcal{T}$ be a non-degenerate block satisfying  {\rm (A1)}. Then the following statements hold.
\begin{enumerate}
\item[(1)] $Z(\mathcal{T})=k1\oplus (\prod_{X\in S} \delta(X))$, where $\prod_{X\in S} \delta(X)$ is an ideal of square zero.
  \item[(2)]  Let $\lambda\in Z(\mathcal{T})$ be invertible. Then $({\rm Id}_\mathcal{T}, \Sigma(\lambda))$ is a triangle functor if and only if $1-\lambda$ lies in  $\prod_{X\in S} \delta(X)$.
\end{enumerate}
\end{prop}

Here, the element $\delta(X)\in Z(\mathcal{T})$ is obtained by applying Lemma \ref{lem:center} to the fixed almost-vanishing endomorphism $\Delta_X$ for each $X\in S$. We observe that if $({\rm Id}_\mathcal{T}, \omega)$ is a triangle functor, then there is a unique invertible element $\lambda\in Z(\mathcal{T})$ with $\omega=\Sigma(\lambda)$. Hence, statement (2) characterizes all such connecting isomorphisms $\omega$.

\begin{proof}
Let $\lambda\in Z(\mathcal{T})$. We assume that $\lambda_X=c(X){\rm Id}_X+b(X)\Delta_X$ for $X\in S$ and $\lambda_Y=c(Y){\rm Id}_Y$ for $Y\in S'$, where $c(X), b(X)$ and $c(Y)$ are scalars. For (1), it suffices to claim that the function $c$ is constant on ${\rm ind}\mathcal{T}$.

 We first observe that if ${\rm Hom}_\mathcal{T}(U, V)\neq 0$ for $U, V\in {\rm ind}\mathcal{T}$, we have $c(U)=c(V)$. We just apply the naturality of $\lambda$ to the nonzero morphism $U\rightarrow V$. Here, if $U$ or $V$ lies in $S$, we use the fact that $\Delta_U$ or $\Delta_V$ is almost-vanishing. Now the claim follows by combining the observation and the following fact: for any indecomposables $U$ and $V$,  there is a sequence $U=X_0, X_1, \cdots, X_n=V$ of indecomposables, such that ${\rm Hom}_\mathcal{T}(X_i, X_{i+1})\neq 0$ or ${\rm Hom}_\mathcal{T}(X_{i+1}, X_i)\neq 0$; see \cite[Proposition 4.2 and Remark 4.7]{CR}.

For (2), it suffices to claim that $({\rm Id}_\mathcal{T}, \Sigma(\lambda))$ is a triangle functor if and only if the constant function $c$ has value $1$.

For the ``only if" part, we take a nonzero non-invertible morphism $h\colon U\rightarrow \Sigma(V)$ for some indecomposables $U$ and $V$, since $\mathcal{T}$ is  non-degenerate. Form an exact triangle  $V\stackrel{f}\rightarrow E\stackrel{g}\rightarrow U \stackrel{h}\rightarrow \Sigma(V)$. Applying the triangle functor $({\rm Id}_\mathcal{T}, \Sigma(\lambda))$ to this triangle, we obtain another exact triangle
$$V\stackrel{f}\longrightarrow E\stackrel{g}\longrightarrow U\xrightarrow{c(V) h} \Sigma(V).$$
Here, we use the fact $\Sigma(\Delta_V)\circ h=0$,  in case $V$ lies in $S$. By Lemma \ref{lem:scalar}, we infer that $c(V)=1$.

For the ``if" part,  we take an arbitrary exact triangle
$$\xi\colon X\stackrel{a}\longrightarrow Y\stackrel{b}\longrightarrow Z\stackrel{u}\longrightarrow \Sigma(X)$$ in $\mathcal{T}$. We observe that $u$ is isomorphic to $u'\oplus {\rm Id}_{\Sigma(K)}$ for a radical morphism $u'\colon Z'\rightarrow \Sigma(X')$ and some object $K$. Then $\xi$ is isomorphic to the direct sum of $\xi'\colon X'\stackrel{a'}\rightarrow Y \stackrel{b'}\rightarrow Z'\stackrel{u'}\rightarrow \Sigma(X')$ and the trivial triangle $\xi''\colon K\rightarrow 0\rightarrow \Sigma(K) \xrightarrow{{\rm Id}_{\Sigma(K)}}  \Sigma(K)$. We apply $({\rm Id}_\mathcal{T}, \Sigma(\lambda))$ to $\xi'$. Then the resulted triangle coincides with $\xi'$, since we have $\Sigma(\lambda_{X'})\circ u'=u'$, using the fact that the morphism $u'$ is radical. The trivial triangle $\xi''$ stays trivial by applying $({\rm Id}_\mathcal{T}, \Sigma(\lambda))$.  It follows that the resulted triangle by applying  $({\rm Id}_\mathcal{T}, \Sigma(\lambda))$ to $\xi$ is exact, that is, $({\rm Id}_\mathcal{T}, \Sigma(\lambda))$ is a triangle functor.
\end{proof}

We need a further condition to determine the triangle center.

\vskip 5pt

\noindent {\rm (A2)} \quad We assume (A1). Moreover, we assume that the action of $\Sigma$ on $S$ is \emph{free}, that is, for each $X\in S$ and $n\geq 1$, $\Sigma^n(X)\simeq X$ implies $n=1$.

\vskip 5pt

We denote by $S/\Sigma$ the set of $\Sigma$-orbits in $S$. Since each orbit is infinite, we can always make the following choice: for each $X\in S$, we choose  an isomorphism
\begin{align}\label{equ:ch1}
t_X\colon \Sigma(X)\longrightarrow X'
 \end{align}
 for a uniquely determined $X'\in S$; moreover, we assume that
 \begin{align}\label{equ:ch2}
 \Delta_{X'}=t_X\circ \Sigma(\Delta_X)\circ {t_X}^{-1}.
 \end{align}
For the latter assumption, we adjust  the almost-vanishing endomorphisms in (A1).

For $X\in S$, we denote by $[X]$ its $\Sigma$-orbit. We denote by $\delta([X])\in Z(\mathcal{T})$, which is uniquely determined by $\delta([X])_Y=0$ for $Y\in {\rm ind}\mathcal{T}\backslash [X]$ and $\delta([X])_{Z}=\Delta_Z$ for $Z\in [X]$. Indeed, by the identification in Proposition \ref{prop:A1}(1), $\delta([X])$ corresponds to the formal sum $\sum_{Z\in [X]} \delta(Z)$.

\begin{prop}\label{prop:A2}
Let $\mathcal{T}$ be a non-degenerate block satisfying  {\rm (A2)}. We keep the choices (\ref{equ:ch1}) and (\ref{equ:ch2}). Then the following statements hold.
\begin{enumerate}
\item[(1)] $Z_\vartriangle(\mathcal{T})=k1\oplus (\prod_{[X]\in S/\Sigma} \delta([X]))$, where $\prod_{[X]\in S/\Sigma} \delta([X])$ is an ideal of square zero.
  \item[(2)]  Each triangle functor $({\rm Id}_\mathcal{T}, \omega)$ is isomorphic to $({\rm Id}_\mathcal{T}, {\rm Id}_\Sigma)$.
\end{enumerate}
\end{prop}

\begin{proof}
Take $\lambda\in Z_\vartriangle(\mathcal{T})$. By Proposition \ref{prop:A1}(1), we may assume $$\lambda=c+\sum_{X\in S}b(X)\delta(X)$$
for $c, b(X)\in k$. For each $Z\in S$, we have $\Sigma(\lambda_Z)=\lambda_{\Sigma(Z)}$. Applying the naturality of $\lambda$ to the isomorphism $t_Z$ in (\ref{equ:ch1}) and using (\ref{equ:ch2}), we infer that $b(Z)=b(Z')$. Then the function $b$ is constant on each $\Sigma$-orbit of $S$. In other words, $\lambda=c+\sum_{[X]\in S/\Sigma} b(X)\delta([X])$. This proves (1).

For (2), we have an invertible element $\lambda\in Z(\mathcal{T})$ satisfying $\omega=\Sigma(\lambda)$; moreover, by Proposition \ref{prop:A1}(2), we may assume that
$$\lambda=1+\sum_{X\in S}b(X)\delta(X)$$ for $b(X)\in k$. For each $X\in S$, we choose a scalar $\phi(X)$ such that $\phi(X)-\phi(X')=b(X)$, where $X'\in S$ is uniquely determined by the isomorphism (\ref{equ:ch1}). This choice of the function $\phi(-)$  is possible, since the $\Sigma$-orbit $[X]$ is infinite.

Set $\eta=1+\sum_{X\in S}\phi(X) \delta(X)\in Z(\mathcal{T})$, which is invertible. We claim that $\eta\colon ({\rm Id}_\mathcal{T}, \omega)\rightarrow ({\rm Id}_\mathcal{T}, {\rm Id}_\Sigma)$ is the required isomorphism.

To prove that $\eta$ is a natural transformation between triangle functors, it suffices to verify $\eta_{\Sigma(U)}\circ \omega_U=\Sigma(\eta_U)$ for each $U\in {\rm ind}\mathcal{T}$. Recall that $\omega=\Sigma(\lambda)$ and that the value of $\delta(X)$ on $X$ is $\Delta_X$ for each $X\in S$. The left square in the following commutative diagram
\[
\xymatrix{
\Sigma(X) \ar[d]_-{\Sigma(\eta_X)} \ar[rrr]^-{\Sigma({\rm Id}_X+b(X) \Delta_X)} &&& \Sigma(X) \ar[d]^-{\eta_{\Sigma(X)}}\ar[r]^-{t_X} & X' \ar[d]^-{\eta_{X'}}\\
\Sigma(X) \ar@{=}[rrr] &&& \Sigma(X) \ar[r]^-{t_X} & X'
}\]
proves $\eta_{\Sigma(X)}\circ \omega_X=\Sigma(\eta_X)$. Here, the commutativity of the outer diagram follows from the choice of the function $\phi(-)$, since by (\ref{equ:ch2}) we have
$$(t_X)^{-1}\circ \eta_{X'}\circ t_X=\Sigma({\rm Id}_X+\phi(X')\Delta_X).$$
The identity $\eta_{\Sigma(Y)}\circ \omega_Y=\Sigma(\eta_Y)$ is clear for each  $Y\in S'$. This completes the proof of the claim.
\end{proof}

\section{The main results}

In this section, we first recall from \cite{CY} the notion of a $\mathbf{K}$-standard category. The  main techical result claims that the category of projective modules over a certain Nakayama algebra is $\mathbf{K}$-standard; see Theorem \ref{thm:A}. Then using the classification result in \cite{BGS}, we deduce that every derived equivalence between derived-discrete algebras of finite global dimension is standard; see Theorem \ref{thm:B}.

Let $k$ be a field and $\mathcal{A}$ be a $k$-linear additive category. We denote by $\mathbf{K}^b(\mathcal{A})$ the bounded homotopy category. A bounded complex $X=(X^n, d_X^n)_{n\in \mathbb{Z}}$ is visualized as $$\cdots \longrightarrow X^n \stackrel{d_X^n} \longrightarrow X^{n+1} \stackrel{d_X^{n+1}}\longrightarrow X^{n+2}\longrightarrow \cdots,$$
where $X^n=0$ for $|n|\gg 0$. The \emph{support} of $X$, denoted by ${\rm supp}(X)$, is the interval $[m, n]$ in $\mathbb{Z}$, where $m$ is the smallest integer with  $X^m\neq 0$, and $n$ the largest integer with $X^n\neq 0$. Two homotopic complexes may have different supports.

The suspension functor $\Sigma$ sends $X$ to $\Sigma(X)$, which is given by $\Sigma(X)^n=X^{n+1}$ and $d_{\Sigma(X)}^n=-d_X^{n+1}$. We identify $\mathcal{A}$ as the full subcategory of $\mathbf{K}^b(\mathcal{A})$ consisting of stalk complexes concentrated on degree zero. Therefore, for each integer $n$, an object $M\in \mathcal{A}$ yields a stalk complex $\Sigma^n(M)$ concentrated on degree $-n$. These complexes form a full subcategory $\Sigma^n(\mathcal{A})$, where $\mathcal{A}=\Sigma^0(\mathcal{A})$.

We have the following $k$-algebra homomorphism
\begin{align}\label{equ:res}
{\rm res} \colon Z_\vartriangle(\mathbf{K}^b(\mathcal{A}))\longrightarrow Z(\mathcal{A}),
\end{align}
sending $\lambda$ to $\lambda|_\mathcal{A}$, the restriction of $\lambda$ to $\mathcal{A}$. This homomorphism has a right inverse, which sends $\mu\in Z(\mathcal{A})$ to $\mathbf{K}^b(\mu)$, the natural extension of $\mu$ on complexes.

The following notions are introduced in \cite[Sections 3 and 4]{CY}. By \cite[Lemma 4.2]{CY}, the current notion of a $\mathbf{K}$-standard category is equivalent to the original one.

\begin{defn}
(1) A $k$-linear triangle endofunctor $(F, \omega)$ on $\mathbf{K}^b(\mathcal{A})$ is a \emph{pseudo-identity},  provided that $F(X)=X$ for each complex $X$, and for each integer $n$ the restriction $F|_{\Sigma^n(\mathcal{A})}\colon \Sigma^n(\mathcal{A})\rightarrow \Sigma^n(\mathcal{A})$ equals the identity functor.\\
(2) The $k$-linear additive category $\mathcal{A}$ is \emph{$\mathbf{K}$-standard} (over $k$),  provided that each pseudo-identity functor is isomorphic to $({\rm Id}_{\mathbf{K}^b(\mathcal{A})}, {\rm Id}_\Sigma)$, the genuine identity functor; it is \emph{strongly $\mathbf{K}$-standard}, if in addition the homomorphism (\ref{equ:res}) is injective.  \hfill $\square$
\end{defn}

We say that a pseudo-identity $(F, \omega)$ on $\mathbf{K}^b(\mathcal{A})$ is \emph{normalized}, provided that $\omega_X=\Sigma({\rm Id}_X)$ for each stalk complex $X$. Here, by a stalk complex, we mean a complex of the form $\Sigma^n(M)$ for some integer $n$ and $M\in \mathcal{A}$. According to the following observation,  up to isomorphism, we may always assume that any pseudo-identity is normalized.

\begin{lem}\label{lem:normal}
Let  $(F, \omega)$ be a pseudo-identity on $\mathbf{K}^b(\mathcal{A})$. Then there  exists a normalized pseudo-identity $(F', \omega')$ with an isomorphism $\delta\colon (F, \omega)\rightarrow (F', \omega')$.
\end{lem}

\begin{proof}
We observe that, for each $n\in \mathbb{Z}$, there is a unique invertible element $\lambda_n\in Z(\mathcal{A})$ such that  for each object $M\in \mathcal{A}$,  the connecting isomorphism
$$\omega_{\Sigma^n(M)}\colon F\Sigma(\Sigma^n M)=\Sigma^{n+1}(M)\longrightarrow \Sigma^{n+1}(M)=\Sigma F(\Sigma^n M)$$
equals $\Sigma^{n+1}((\lambda_n)_M)$. Take a sequence $\{a_n\}_{n\in \mathbb{Z}}$ of invertible elements in $Z(\mathcal{A})$ such that $a_0=1$ and $\lambda_n=a_{n}^{-1}a_{n+1}$ for each integer $n$.

For each complex $X$, we take an isomorphism $\delta_X\colon F(X)=X\rightarrow X=F'(X)$ such that $\delta_{\Sigma^n(M)}=\Sigma^n((a_n)_M)$ for each integer $n$ and $M\in \mathcal{A}$. We use  $\delta_X$'s as the adjusting isomorphisms to obtain the required isomorphism; see Lemma \ref{lem:adjust}. By the construction of $(F', \omega')$ in the proof, we infer that $(F', \omega')$ is a normalized pseudo-identity.  Here, we recall that $\omega'=(\Sigma\delta)\circ \omega\circ (\delta\Sigma)^{-1}$, from which and the choice of the sequence $\{a_n\}_{n\in \mathbb{Z}}$ we deduce that $\omega'_{\Sigma^n(M)}$ is the identity.
\end{proof}

Let $A$ be a finite dimensional $k$-algebra. We denote by $A\mbox{-mod}$ the abelian category of finite dimensional left $A$-modules, and by $A\mbox{-proj}$ the full subcategory formed by projective $A$-modules. Denote by $\mathbf{D}^b(A\mbox{-mod})$ the bounded derived category.

Two algebras $A$ and $B$ are \emph{derived equivalent} provided that there is a $k$-linear triangle equivalence $(F, \omega)\colon \mathbf{D}^b(A\mbox{-mod})\rightarrow \mathbf{D}^b(B\mbox{-mod})$, called a \emph{derived equivalence}. For any $B$-$A$-bimodule, we always require that $k$ acts centrally. A bounded complex $X$ of $B$-$A$-bimodules is a called \emph{two-sided tilting complex} provided that the derived tensor functor $X\otimes_A^\mathbb{L}-\colon \mathbf{D}^b(A\mbox{-mod})\rightarrow \mathbf{D}^b(B\mbox{-mod})$ is an equivalence. A derived equivalence $(F, \omega)$ is \emph{standard},  if it is isomorphic to $X\otimes_A^\mathbb{L}-$, as a triangle functor, for some two-sided tilting complex $X$. It is an open question in \cite{Ric} whether all derived equivalences are standard. For details on derived equivalences, we refer to \cite{Zim}.

The following result is the main motivation to study the $\mathbf{K}$-standardness, which is obtained by combining \cite[Theorems 6.1 and 5.10]{CY}.

\begin{prop}\label{prop:stand}
Let $A$ and $B$ be two finite dimensional $k$-algebras such that $A \mbox{-{\rm proj}}$ is $\mathbf{K}$-standard. Then any derived equivalence $\mathbf{D}^b(A\mbox{-{\rm mod}})\rightarrow \mathbf{D}^b(B\mbox{-{\rm mod}})$ is standard. \hfill $\square$
\end{prop}

The following observation will be needed.

\begin{lem}\label{lem:inv}
Let $A$ and $B$ be two finite dimensional $k$-algebras of finite global dimension, which are derived equivalent. Then $A \mbox{-{\rm proj}}$ is $\mathbf{K}$-standard if and only if so is $B\mbox{-{\rm proj}}$.
\end{lem}

\begin{proof}
By \cite[Section 6]{CY}, $A \mbox{-{\rm proj}}$ is $\mathbf{K}$-standard  if and only if the module category $A\mbox{-mod}$ is $\mathbf{D}$-standard in the sense of \cite[Definition 5.1]{CY}. Then the desired equivalence follows from \cite[Lemma 5.12]{CY}.
\end{proof}

Let $r\geq 1$ and $N\geq r$. Let $A=A(r, N)$ be the finite dimensional algebra given by the following cyclic quiver
\[
\xymatrix@R=6pt{
&  1 \ar[r]^-{\alpha_1} & \cdots  \ar[r]^-{\alpha_{r-2}} & r-1 \ar[dr]^-{\alpha_{r-1}}\\
0\ar[ur]^-{\alpha_0} &&&& r\ar[dl]^-{\alpha_r}\\
& N-1 \ar[ul]^-{\alpha_{N-1}} & \cdots \ar[l]^-{\alpha_{N-2}} & r+1 \ar[l]^-{\alpha_{r+1}}
}\]
with $r$ relations $\{\alpha_0\alpha_{N-1}, \alpha_1\alpha_0, \cdots, \alpha_{r-1}\alpha_{r-2}\}$. Here, we write the concatenation of arrows from right to left. Then $A$ is a Nakayama algebra. For each vertex $i$, we denote by $e_i$ the corresponding primitive idempotent.

We assume that $r<N$. Then the global dimension of $A$ is $r+1$. For $0\leq i\leq r-1$, we denote by $P_i=Ae_i$ the indecomposable projective module corresponding to $i$. Similarly, we have $Q_a=Ae_{a}$ for $r\leq a< N$. We observe that $\{P_0, P_1, \cdots, P_{r-1}\}$ is a complete set of representatives  of isoclasses of  indecomposable projective-injective $A$-modules.

The following convention of unnamed arrows will be used later.

\vskip 5pt

\noindent {\bf Convention} $(\dag)$ \quad  For $1\leq i\leq r-1$, we denote by the unnamed arrow $P_i\rightarrow P_{i-1}$ the unique $A$-module homomorphism sending $e_i$ to $\alpha_{i-1}$. The unnamed arrow $P_0\rightarrow P_{r-1}$ sends $e_0$ to the path $\alpha_{N-1}\cdots \alpha_{r}\alpha_{r-1}$. Indeed, we will understand the index $i$ in the cyclic group $\mathbb{Z}/{r\mathbb{Z}}$. For $r\leq a< N$, we denote by the unnamed arrows $Q_a\rightarrow P_{r-1}$ and $P_0\rightarrow Q_a$ the unique homomorphisms sending $e_a$ to $\alpha_{a-1}\cdots \alpha_r\alpha_{r-1}$, and $e_0$ to $\alpha_{N-1}\cdots \alpha_{a+1}\alpha_a$, respectively. For $r\leq b<a<N$, the unnamed arrow $Q_a\rightarrow Q_b$ sends $e_a$ to the unique path going from $b$ to $a$.

\vskip 5pt

The main result of this paper is as follows.

\begin{thm}\label{thm:A}
Let $1\leq r\leq N$ and $A=A(r, N)$ be the above algebra. Then the following statements hold.
\begin{enumerate}
\item If $r=1$, then the category $A\mbox{-{\rm proj}}$ is $\mathbf{K}$-standard, but not strongly $\mathbf{K}$-standard.
\item If $r>1$, then the category $A\mbox{-{\rm proj}}$ is strongly $\mathbf{K}$-standard.
\end{enumerate}
Consequently, any derived equivalence  $\mathbf{D}^b(A(r, N)\mbox{-{\rm mod}})\rightarrow \mathbf{D}^b(B\mbox{-{\rm mod}})$ is standard.  \hfill $\square$
\end{thm}

For the proof, the last statement follows from Proposition \ref{prop:stand}. The case $r=N$ is treated in \cite[Section 7]{CY}. Assume that $r<N$. We will prove  (1) in Section 4 and (2) in Section 5.

For the rest of this section, we assume that the base field $k$ is algebraically closed. We will apply Theorem \ref{thm:A} to derived-discrete algebras of finite global dimension.

 Recall from \cite{Vos} that a finite dimensional $k$-algebra $A$ is \emph{derived-discrete},  provided that for each vector ${\bf n}=(n_i)_{i\in \mathbb{Z}}$ of natural numbers there are only finitely many isomorphism classes of indecomposable objects in $\mathbf{D}^b(A\mbox{-mod})$  with cohomology dimension vector $\bf{n}$. Here, for a complex $X$ of $A$-modules, its cohomology dimension vector is given by $({\rm dim}\; H^i(X))_{i\in \mathbb{Z}}$. By \cite[Proposition 1.1]{Vos}, derived-discrete algebras are closed under derived equivalences.

By the classification result  \cite{BGS} of derived-discrete algebras up to derived equivalence, we have the following result.

\begin{thm}\label{thm:B}
Let $k$ be an algebraically closed field, and  $A$ be a derived-discrete $k$-algebra of finite global dimension. Then the category $A\mbox{-{\rm proj}}$ is $\mathbf{K}$-standard. Consequently, any derived equivalence between derived-discrete algebras of finite global dimension is standard.
\end{thm}

\begin{proof}
The second statement follows from Proposition \ref{prop:stand}. Without loss of generality, we may assume that the algebra $A$ is connected.

For the first statement, it suffices to claim that $A$ is derived equivalent to a triangular algebra $C$, or the algebra $A(r, N)$ for $1\leq r<N$. By \cite[Proposition 5.13]{CY} and Theorem \ref{thm:A}, we infer that for both $C$ and $A(r, N)$, their categories of projective modules are $\mathbf{K}$-standard. Then we are done by Lemma \ref{lem:inv}.

For the claim, we recall from \cite[Theorem 2.1]{Vos} and \cite[Proposition 2.3]{BGS} that $A$ is either derived equivalent to the path algebra of a Dynkin quiver, or to the algebra $A(r, N, m)$ given by the following quiver
\[
\xymatrix@!=12pt{
& & & & &  1 \ar[r]^-{\alpha_1} & \cdots  \ar[r]^-{\alpha_{r-2}} & r-1 \ar[dr]^-{\alpha_{r-1}}\\
-m & 1-m  \ar[l]_{\beta_{m-1}} & \cdots \ar[l] & -1 \ar[l]_{\beta_1} &  0 \ar[l]_{\beta_0}\ar[ur]^-{\alpha_0} &&&& r\ar[dl]^-{\alpha_r}\\
& & & & & N-1 \ar[ul]^-{\alpha_{N-1}} & \cdots \ar[l]^-{\alpha_{N-2}} & r+1 \ar[l]^-{\alpha_{r+1}}
}\]
with $r$ relations $\{\alpha_0\alpha_{N-1}, \alpha_1\alpha_0, \cdots, \alpha_{r-1}\alpha_{r-2}\}$, where $1\leq r<N$ and $m\geq 0$. We observe that $A(r, N, 0)=A(r, N)$. If $m>0$, the algebra $A(r, N, m)$ is derived equivalent to a triangular algebra; see \cite[Proof 2.4 (h)]{BGS}. This proves the claim.
\end{proof}

\begin{rem}
(1) Recall from \cite{KYan,BPP} that the group of standard autoequivalences on $\mathbf{D}^b(A\mbox{-mod})$ is explicitly calculated, where $A$ is a derived-discrete algebra of finite global dimension. Theorem \ref{thm:B} implies that this group coincides with the whole group of (triangle) autoequivalences on $\mathbf{D}^b(A\mbox{-mod})$.

(2) It is natural to expect that Theorem \ref{thm:B} also holds for a derived-discrete algebra of infinite global dimension. By \cite[Proposition 2.3]{BGS}, it suffices to prove that $A(r,r, m)\mbox{-proj}$ is $\mathbf{K}$-standard for $m\geq 1$. However, the argument in this paper does not apply to the algebras $A(r, r, m)$. It seems that the most difficult case is the algebra $A(1, 1, m)$.
\end{rem}

\section{The case $r=1$}

In this section, we prove Theorem \ref{thm:A} for the case $r=1$ and $N>1$.

Let $A=A(1, N)$, $\mathcal{A}=A\mbox{-proj}$ and $\mathcal{T}=\mathbf{K}^b(\mathcal{A})$.  We will use Convention ($\dag$) in Section 3. Consider the following objects in $\mathcal{T}$:
\begin{align*}
X_{m, n}& =\cdots \rightarrow 0\rightarrow P_0\rightarrow P_0\rightarrow \cdots \rightarrow P_0\rightarrow P_0\rightarrow 0\rightarrow \cdots\\
L_{m, n, a} &= \cdots \rightarrow 0\rightarrow Q_a\rightarrow P_0\rightarrow P_0\rightarrow \cdots \rightarrow P_0\rightarrow P_0 \rightarrow 0\rightarrow \cdots\\
R_{m, n, b} &= \cdots \rightarrow 0\rightarrow P_0\rightarrow P_0\rightarrow \cdots \rightarrow P_0\rightarrow P_0 \rightarrow Q_b \rightarrow 0\rightarrow \cdots\\
B_{m, n, a, b} &=\cdots \rightarrow 0\rightarrow Q_a\rightarrow P_0\rightarrow P_0\rightarrow \cdots \rightarrow P_0\rightarrow P_0 \rightarrow Q_b \rightarrow 0 \rightarrow \cdots\\
Z_{m, a, b} &=\cdots \rightarrow 0\rightarrow Q_a\rightarrow Q_b \rightarrow 0\rightarrow \cdots.
\end{align*}
Some comments on the lower indices are necessary. The supports of $X_{m, n}$, $L_{m, n, a}$ and $R_{m, n, b}$ are $[m, n]$ for $m\leq n$, where $a, b$ run from $1$ to $N-1$. We observe that $X_{m, m}=\Sigma^{-m}(P_0)$ and $L_{m, m,a}=R_{m, m, a}=\Sigma^{-m}(Q_a)$. For $m\leq n-2$ and $1\leq a, b< N$, the complex $B_{m, n, a, b} $ is defined, whose support is $[m, n]$. The complex $Z_{m, a, b}$ is supported in $[m, m+1]$, where $1\leq b<a< N$. In particular, the complex $Z_{m, a, b}$ is not defined for the case $N=2$.

\begin{lem}\label{lem:ind}
Keep the above notation. Then
$$\Lambda=\{X_{m, n}, L_{m, n, a}, R_{m, n, b}, B_{m, n, a, b}, Z_{m, a, b}\}$$
is a complete set of representatives  of isoclasses of indecomposable objects in $\mathcal{T}$.
\end{lem}

\begin{proof}
By \cite[Theorem 3]{BM},  the isoclasses of  indecomposable objects in $\mathcal{T}$, up to suspensions,  are in a bijection to the generalized strings in $A$. For $1\leq a, b< N$ and $p\geq 0$, we consider the following generalized string
$$w_{a, b}^p:=w_0w_1\cdots w_p w_{p+1}=(\alpha_b\cdots \alpha_1\alpha_0) (\alpha_{N-1}\cdots \alpha_1\alpha_0)^p  (\alpha_{N-1}\cdots \alpha_{a+1}\alpha_a),$$
where $w_0=\alpha_b\cdots \alpha_1\alpha_0$, $w_1=w_2=\cdots =w_p=\alpha_{N-1}\cdots \alpha_1\alpha_0$ and $w_{p+1}=\alpha_{N-1}\cdots \alpha_{a+1}\alpha_a$ are nonzero paths in $A$. It is clear that any generalized string $w$ is a  generalized substring of  $w_{a, b}^p$, that is, of the form $w_iw_{i+1}\cdots w_j$ for some $0\leq i\leq j\leq p+1$. Then the indecomposable object $P_w$ corresponding to $w$ is isomorphic to an element in $\Lambda$; compare \cite[Definition 2]{BM}. For example, up to suspensions, the complex $B_{m, n, a, b}$ corresponds to the generalized string $w_{a, b}^{n-m-2}$.
\end{proof}

\begin{rem}\label{rem:alm}
We compute the endomorphism algebras of objects in $\Lambda$. We have ${\rm End}_\mathcal{T}(X_{m, n})=k{\rm Id}_{X_{m, n}}\oplus k\Delta_{m, n}$, where $\Delta_{m, n}$ is of the following form.
\[\xymatrix{
0\ar[r] & P_0 \ar[d]\ar[r] & P_0 \ar[d]_{0}\ar[r] & \cdots \ar[r] & P_0 \ar[d]_{0}\ar[r] & P_0\ar[d]_{0}\ar[r]  & 0\\
 0\ar[r] & P_0\ar[r]  & P_0\ar[r] &  \cdots \ar[r] & P_0\ar[r]  & P_0\ar[r] &  0
}\]
Here, the unique nonzero vertical map is the unnamed arrow $P_0\rightarrow P_0$, sending $e_0$ to $\alpha_{N-1}\cdots\alpha_1\alpha_0$.  We observe that the endomorphism $\Delta_{m, n}$ is almost-vanishing, since $\nu(X_{m, n})\simeq X_{m, n}$ with $\nu$ the Nakayama functor; see \cite[I.4.6]{Hap88}. For other objects $Y$ in $\Lambda$, we have ${\rm End}_\mathcal{T}(Y)=k{\rm Id}_Y$.

In particular, the category $\mathcal{T}$ satisfies the condition (A2) in Section 2, where $S$ is the subset of $\Lambda={\rm ind}\mathcal{T}$ consisting of the complexes $X_{m, n}$. Moreover, the isomorphism (\ref{equ:ch1}) for $X_{m, n}$ is given by
$$t_{m, n}\colon \Sigma(X_{m, n}) \longrightarrow X_{m-1, n-1},$$
where its $p$-the component  $t_{m, n}^p$ is by multiplying $(-1)^p$.  Then we have $(\ref{equ:ch2})$ for $X_{m, n}$, that is,
$$\Delta_{m-1, n-1}=t_{m, n}\circ \Sigma(\Delta_{m, n})\circ t_{m, n}^{-1}.$$
\end{rem}

In what follows, we will describe  morphisms between the objects in $\Lambda$. One might find these morphisms in \cite{ALP}, where a combinatorial description of morphisms between the indecomposables in the bounded derived category of gentle algebras is available. We try to make this exposition self-contained, and use the notation for our convenience. We divide the morphisms between objects in $\Lambda$ into four types: \emph{inclusions}, \emph{projections}, \emph{connections} and morphisms of the \emph{mixed} type.

 The first type is induced by the obvious inclusion: $i_{m, n}\colon X_{m, n}\rightarrow X_{m-1, n}$, $j_{m, n,a}\colon X_{m, n}\rightarrow L_{m-1, n, a}$, $i'_{m, n, b}\colon R_{m, n, b}\rightarrow R_{m-1, n, b}$ for $m\leq n$ and $1\leq a, b< N$; $\iota_{m, n, a, b}\colon R_{m, n, b}\rightarrow B_{m-1, n, a, b}$ for $m<n$ and $1\leq a, b< N$; $\xi_{m, a, b}\colon L_{m, m, b}\rightarrow Z_{m-1, a, b}$ for each integer  $m$ and $1\leq b<a<N$. We will call any composition of these morphisms an \emph{inclusion}, which will be denoted by ``inc".

The second type is induced by the obvious projection: $\pi_{m, n}\colon X_{m, n}\rightarrow X_{m, n-1}$, $\pi'_{m, n, a}\colon L_{m, n, a}\rightarrow L_{m, n-1, a}$ and $p_{m, n, b}\colon R_{m, n, b}\rightarrow X_{m, n-1}$ for $m<n$ and $1\leq a, b < N$; $q_{m, n, a, b}\colon B_{m, n, a, b}\rightarrow L_{m, n-1, a}$ for $m\leq n-2$ and $1\leq a, b < N$; $\zeta_{m, a, b}\colon Z_{m, a, b}\rightarrow L_{m, m, a}$ for each integer  $m$ and $1\leq b<a< N$. We will call any composition of these morphisms a \emph{projection}, which will be denoted by ``pr".

 For the third type, we denote by $c_{l, m, n, a, b}\colon L_{l, m, a}\rightarrow R_{m, n, b}$, for $l\leq m \leq n$ and $1\leq a, b< N$,  the following morphism
\[\xymatrix{
0\ar[r] & Q_a\ar[r] & P_0\ar[r] & \cdots \ar[r] & P_0\ar[d]\ar[r]  & 0\\
&             &              &               0\ar[r] & P_0\ar[r] &  \cdots \ar[r] &  P_0\ar[r] & Q_b\ar[r] &  0.
}\]
Here, as usual, the unique nonzero vertical map $P_0\rightarrow P_0$ sends $e_0$ to the path $\alpha_{N-1}\cdots \alpha_1\alpha_0$. If $l=m<n$, the vertical map is given by the unnamed arrow $Q_a\rightarrow P_0$; if $l<m=n$, it is given by $P_0\rightarrow Q_b$. If $l=m=n$, we have to assume that $b<a$, in which case we have $c_{l, l, l, a, b}=\Sigma^{-l}(\phi)$, where $\phi\colon Q_a\rightarrow Q_b$ is the unnamed arrow. We will call any morphism of the form $${\rm pr}\circ c_{l, m, n, a, b}\circ {\rm inc}$$
 a \emph{connection}. For example, the following morphism $c_{l, m, n}\colon X_{l, m}\rightarrow X_{m, n}$
\[\xymatrix{
0\ar[r] & P_0\ar[r] & P_0\ar[r] & \cdots \ar[r] & P_0\ar[d]\ar[r]  & 0\\
&             &              &               0\ar[r] & P_0\ar[r] &  \cdots \ar[r] &  P_0\ar[r] & P_0\ar[r] &  0.
}\]
is a connection for $l\leq m\leq n$. We observe that the composition
$$X_{m, n}\stackrel{\rm pr}\longrightarrow  X_{m, m} \xrightarrow{c_{m, m, n}} X_{m, n}$$
equals the almost-vanishing endomorphism $\Delta_{m, n}$ in Remark \ref{rem:alm}.

The fourth type, called the \emph{mixed} type, is  divided into $11$ classes of morphisms. They are listed as follows. Here,  ``mx" stands for the word ``mixed".
\[
\xymatrix@!=8pt{
L_{m, n, a}\ar[d]_{{\rm (mx.I)}}\colon &  0\ar[r] & Q_a \ar[r] & P_0\ar[r] & \cdots \ar[r] & P_0\ar[d] \ar[r] &P_0 \ar@{=}[d]\ar[r] & \cdots \ar[r] & P_0 \ar@{=}[d] \ar[r]  & 0 \\
L_{m', n, b}\colon & &  & & 0\ar[r] & Q_b \ar[r] & P_0\ar[r] & \cdots \ar[r] & P_0\ar[r] & 0,
}\]
for $m<m'<n$ and $1\leq a, b < N$;
\[
\xymatrix@!=8pt{
L_{m, n, a}\ar[d]_{{\rm (mx.II)}}\colon &  0\ar[r] & Q_a \ar[d]\ar[r] & P_0\ar@{=}[d]\ar[r] & \cdots \ar[r]  & P_0 \ar@{=}[d] \ar[r]  & 0 \\
L_{m, n, b}\colon & 0\ar[r] & Q_b \ar[r] & P_0\ar[r] & \cdots \ar[r]  & P_0 \ar[r]  & 0,
}\]
for $m<n$ and $1\leq b<a < N$;
\[
\xymatrix@!=8pt{
R_{m, n, a}\ar[d]_{{\rm (mx.III)}}\colon &  0\ar[r] & P_0 \ar@{=}[d]\ar[r] &  \cdots \ar[r]& P_0\ar@{=}[d]\ar[r] & Q_a\ar[d] \ar[r]  & 0 \\
R_{m, n', b}\colon & 0\ar[r] & P_0 \ar[r] & \cdots \ar[r] & P_0\ar[r] & P_0\ar[r] & \cdots \ar[r]  & P_0 \ar[r]  & Q_b\ar[r] &  0,
}\]
for $m<n<n'$ and $1\leq a, b < N$;
\[
\xymatrix@!=8pt{
R_{m, n, a}\ar[d]_{{\rm (mx.IV)}}\colon &  0\ar[r] & P_0 \ar@{=}[d]\ar[r] &  \cdots \ar[r]& P_0\ar@{=}[d]\ar[r] & Q_a\ar[d] \ar[r]  & 0 \\
R_{m, n, b}\colon & 0\ar[r] & P_0 \ar[r] & \cdots \ar[r]  & P_0 \ar[r]  & Q_b\ar[r] &  0,
}\]
for $m<n$ and $1\leq b<a < N$;
\[
\xymatrix @C=10pt{
B_{m, n, a, b} \ar[d]_{{\rm (mx.V)}} \colon & 0 \ar[r] & Q_a \ar[r]  & \cdots \ar[r] & P_0\ar[d]\ar[r] & P_0\ar@{=}[d]\ar[r] & \cdots \ar[r] & P_0\ar@{=}[d]\ar[r] &  Q_b \ar[d]\ar[r] & 0 \\
B_{m', n', a', b'} \colon  & & & 0 \ar[r] & Q_{a'} \ar[r] & P_0 \ar[r] & \cdots \ar[r] & P_0 \ar[r] & P_0\ar[r] & \cdots \ar[r] &  Q_{b'} \ar[r] & 0,
}\]
for $m<m'<n<n'$ and $1\leq a, b, a', b'< N$;
\[
\xymatrix@!=8pt{
B_{m, n, a, b} \ar[d]_{{\rm (mx.VI)}} \colon & 0 \ar[r] & Q_a\ar[d] \ar[r] & P_0\ar@{=}[d] \ar[r] & \cdots \ar[r] & P_0\ar@{=}[d]\ar[r] &  Q_b \ar[d]\ar[r] & 0 \\
B_{m, n', a', b'} \colon  &  0 \ar[r] & Q_{a'} \ar[r] & P_0 \ar[r] & \cdots \ar[r] & P_0  \ar[r]  & P_0\ar[r] & \cdots \ar[r] & Q_{b'} \ar[r] & 0,
}\]
for $m+2\leq n<n'$ and $1\leq a, b, a', b'< N$ with $a'\leq a$, where for the case $a=a'$ the unnamed arrow $Q_a\rightarrow Q_{a'}$ is the identity map; dually, we have
\[
\xymatrix@!=8pt{
B_{m, n, a, b} \ar[d]_{{\rm (mx.VII)}} \colon & 0 \ar[r] & Q_a \ar[r]  & \cdots \ar[r] & P_0\ar[d]\ar[r] & P_0\ar@{=}[d]\ar[r] & \cdots \ar[r] & P_0\ar[r] \ar@{=}[d]&Q_b \ar[d]\ar[r] & 0 \\
B_{m', n, a', b'} \colon  & & & 0 \ar[r] & Q_{a'} \ar[r] & P_0 \ar[r] & \cdots \ar[r] & P_0\ar[r] &  Q_{b'}\ar[r] & 0,
}\]
for $m<m'\leq n-2$ and $1\leq a, b, a', b'< N$ with $b'\leq b$, where for the case $b=b'$ the unnamed arrow $Q_b\rightarrow Q_{b'}$ is the identity map;
\[
\xymatrix@!=8pt{
B_{m, n, a, b} \ar[d]_{{\rm (mx.VIII)}} \colon & 0 \ar[r] & Q_a\ar[d]  \ar[r] & P_0\ar@{=}[d]\ar[r] &\cdots \ar[r] & P_0\ar@{=}[d]\ar[r]  & Q_b \ar[d]\ar[r] & 0 \\
B_{m, n, a', b'} \colon  &  0 \ar[r] & Q_{a'} \ar[r] & P_0 \ar[r] & \cdots \ar[r] & P_0\ar[r] &  Q_{b'}\ar[r] & 0,
}\]
for $m\leq n-2$, $a'\leq a$ and $b'\leq b$;
\[
\xymatrix@!=8pt{
B_{m, n, a, b} \ar[d]_{{\rm (mx.IX)}} \colon & 0 \ar[r] & Q_a  \ar[r] & P_0\ar[r] &\cdots \ar[r] & P_0\ar[d]\ar[r]  & Q_b \ar[d]\ar[r] & 0 \\
Z_{n-1, a', b'} \colon  &   & & & 0 \ar[r] & Q_{a'} \ar[r] &   Q_{b'}\ar[r] & 0,
}\]
for $m\leq n-2$ and $1\leq b'\leq b<a'< N$;
\[
\xymatrix@!=8pt{
Z_{m, a, b} \ar[d]_{{\rm (mx.X)}}\colon &  0 \ar[r] & Q_a \ar[d] \ar[r] & Q_b \ar[d] \ar[r] & 0\\
B_{m, n, a', b'}\colon & 0 \ar[r] & Q_{a'} \ar[r] & P_0 \ar[r] & \cdots \ar[r] & P_0 \ar[r] & Q_{b'} \ar[r] & 0
}\]
for $m\leq n-2$ and $1\leq b<a'\leq a<N$;
\[
\xymatrix@!=8pt{
Z_{m, a, b} \ar[d]_{{\rm (mx.XI)}}\colon &  0 \ar[r] & Q_a \ar[d] \ar[r] & Q_b \ar[d] \ar[r] & 0\\
Z_{m,  a', b'}\colon & 0 \ar[r] & Q_{a'} \ar[r]  & Q_{b'} \ar[r] & 0
}\]
for each integer $m$ and $1\leq b'\leq b<a'\leq a< N$.

We mention that a morphism of class ${\rm (mx.V)}$ is zero provided that $n=m'+1$ and $a'\leq b$. The morphisms of classes ${\rm (mx.VI)}$, ${\rm (mx.VII)}$ and ${\rm (mx.VIII)}$ might be viewed as degenerate cases of the ones of class ${\rm (mx.V)}$, by requiring that $n=n'$ or $m=m'$.

The first statement of the following result might be deduced from \cite{ALP}, while the second is essentially due to \cite[Proposition 6.2]{BPP}.

\begin{prop}\label{prop:A}
Keep the notation as above. Then the following statements hold.
\begin{enumerate}
\item
The above morphisms \begin{align*}
&\{i_{m, n}, j_{m, n, a}, i'_{m, n, b}, \iota_{m, n, a, b}, \xi_{m, a, b}\}\cup \{\pi_{m, n},\pi'_{m, n, a}, p_{m, n, b}, q_{m, n, a,b}, \zeta_{m, a,b} \}\\
&\cup \{c_{l, m, n, a, b}\}\cup \{{\rm (mx.I)-(mx.XI)}\}\end{align*}
form a spanning set of  $\mathcal{T}$.
\item For two objects $X, Y\in \Lambda$, we have ${\rm dim} \; {\rm Hom}_\mathcal{T}(X, Y)\leq 2$; moreover, the equality holds if and only if $X=Y=X_{m, n}$.
\end{enumerate}
\end{prop}

We make some preparations for the proof.  The following terminology will be convenient. For two objects $X, Y\in \Lambda$ with ${\rm supp}(X)=[m, n]$ and ${\rm supp}(Y)=[p,q]$, a nonzero morphism $f\colon X\rightarrow Y$ is said to be \emph{crossing} provided that for any morphism $g\colon X\rightarrow Y$, which is homotopic to $f$, satisfies $g^n\neq 0$ and $g^p\neq 0$. It follows that the pair $([m, n], [p, q])$ of intervals is necessarily \emph{crossing}, meaning that $m\leq p\leq n\leq q$. We call the positive integer $n-p+1$ the \emph{width} of $f$.

\begin{lem}\label{lem:A}
For any given objects $X, Y\in \Lambda$, there exist two objects $X'$ and $Y'$ in $\Lambda$ with an inclusion ${\rm inc}\colon Y'\rightarrow Y$ and a projection ${\rm pr}\colon X\rightarrow X'$ such that the pair $({\rm supp}(X'), {\rm supp}(Y'))$ is crossing and that the map is surjective
$${\rm Hom}_\mathcal{T}(X', Y')\longrightarrow {\rm Hom}_\mathcal{T}(X, Y), \quad f\mapsto {\rm inc}\circ f\circ {\rm pr}.$$
\end{lem}

\begin{proof}
We may assume that ${\rm Hom}_\mathcal{T}(X, Y)\neq 0$ and that $({\rm supp}(X), {\rm supp}(Y))$ is not crossing. Assume that ${\rm supp}(X)=[m, n]$ and ${\rm supp}(Y)=[p,q]$. Then there are three cases: $p<m\leq n<q$, or $m<p\leq q<n$, or $p<m\leq q<n$. In each case, we truncate $X$ or $Y$ to obtain the desired $X'$ or $Y'$.
\end{proof}

\begin{lem}
Assume that $X, Y\in \Lambda$ satisfy that the pair $({\rm supp}(X), {\rm supp}(Y))$ is  crossing. Let $f\colon X\rightarrow Y$ be a nonzero morphism in $\mathcal{T}$. Then there exist two objects $X'$ and $Y'$ in $\Lambda$ with an inclusion ${\rm inc}\colon Y'\rightarrow Y$, a projection ${\rm pr}\colon X\rightarrow X'$ and a crossing morphism $f'\colon X'\rightarrow Y'$ satisfying $f={\rm inc}\circ f'\circ {\rm pr}$ in $\mathcal{T}$.
\end{lem}

\begin{proof}
If $f$ is already crossing, we take $X'=X$ and $Y=Y'$. Otherwise, we assume that ${\rm supp}(X)=[m, n]$ and ${\rm supp}(Y)=[p,q]$. Up to homotopy, we have $f^n=0$ or $f^p=0$. In the first case, we truncate $X$ to obtain $X'$; in the latter case, we truncate $Y$. Then we are done by induction on the width.
\end{proof}

We now sketch a proof for Proposition \ref{prop:A}.

\begin{proof}
For (1), we apply the above two lemmas. It suffices to prove the following claim: up to nonzero scalars,  any crossing morphism $f\colon X\rightarrow Y$ is a composition of the listed morphisms. If the width of $f$ is one, then up to nonzero scalars, $f$ is a connection, and then we are done. We may assume that the width of $f$ is at least two. We have to analyse such morphisms between objects in $\Lambda$ case by case; since $\Lambda$ is divided into $5$ classes, we have here $25$ cases to verify the claim. But each case is easy to verify.  We omit the details.

For (2), we apply Lemma \ref{lem:A}. We may assume that the pair $({\rm supp}(X), {\rm supp}(Y))$ is crossing. It suffices to claim  ${\rm dim} \;{\rm Hom}_\mathcal{T}(X, Y)\leq 2$, and the equality holds if and only if $X=Y=X_{m, n}$. The claim implies the desired results, since ${\rm End}_\mathcal{T}(X_{m, n})=k{\rm Id}_{X_{n, m}}\oplus k\Delta_{m, n}$ with $\Delta_{m, n}$ an almost-vanishing morphism.  The proof of the claim is done by computing the Hom-spaces ${\rm Hom}_\mathcal{T}(X, Y)$ for the $25$ cases, under the assumption that $({\rm supp}(X), {\rm supp}(Y))$ is crossing.
\end{proof}

We are in a position to prove Theorem \ref{thm:A} for the case $r=1$. We will often use Lemma \ref{lem:adjust} to adjust the pseudo-identities on $\mathcal{T}$.

In what follows, if we have $F(f)=f$ for an endofunctor $F$ on $\mathcal{T}$ and a   morphism $f$ in $\mathcal{T}$, then we say that $F$ acts trivially on $f$, or acts on $f$ by the identity.

\begin{proof}
Let $(F, \omega)$ be a pseudo-identity on $\mathcal{T}$. We observe first that $F$ acts on the morphisms in Proposition \ref{prop:A}(1) by nonzero scalars.

We will adjust $F$ by nonzero scalars to obtain a new pseudo-identity $(\bar{F}, \bar{\omega})$ such that $\bar{F}$ acts trivially on the spanning set in Proposition \ref{prop:A}(1). By Lemma \ref{lem:span}, as a triangle functor,  $(\bar{F}, \bar{\omega})$ is isomorphic to $({\rm Id}_\mathcal{T}, \omega')$ for some natural automorphism $\omega'$ on $\Sigma$.

By Remark \ref{rem:alm},  Proposition \ref{prop:A2} applies to $\mathcal{T}$. By Proposition \ref{prop:A2}(2), the triangle functor  $({\rm Id}_\mathcal{T}, \omega')$ is isomorphic to the identity functor. This proves that, as a triangle functor, the given $(F, \omega)$ is isomorphic to the identity functor. Then $\mathcal{A}$ is $\mathbf{K}$-standard. In view of Proposition \ref{prop:A2}(1), the homomorphism (\ref{equ:res}) is not injective, since the $\Sigma$-orbit set $S/\Sigma$ is infinite. This also follows from \cite[Proposition 4.5]{Bo}. Here, we use the fact that the center $Z(\mathcal{A})$ is isomorphic to the center of the algebra $A$, which is isomorphic to the algebra of dual numbers.  Then the category $\mathcal{A}$ is not strongly $\mathbf{K}$-standard. We are done.

In what follows, we will adjust the given pseudo-identity $(F, \omega)$ in Step 1. In Step 2-4, we show that the adjusted pseudo-identity acts trivially on the spanning set  in Proposition \ref{prop:A}(1). By Lemma \ref{lem:normal}, we may assume that the given pseudo-identity $(F, \omega)$ is normalized.

\emph{Step 1}\;  For each $X\in \Lambda$ with ${\rm supp}(X)=[m, n]$, we denote by ${\rm inc}_X\colon \Sigma^{-n}(X^n)\rightarrow X$ the obvious inclusion. If $m=n$, ${\rm inc}_X$ is the identity map. We may assume that $F({\rm inc}_X)=\phi(X) {\rm inc}_X$ with $\phi(X)$ a nonzero scalar.

 For each complex $Y$ in $\mathcal{T}$, we define an automorphism $\delta_Y\colon Y\rightarrow Y$ such that $\delta_X=\phi(X)^{-1}{\rm Id}_X$ for $X\in \Lambda$ and $\delta_Z={\rm Id}_Z$ for stalk complexes $Z$. Using $\delta$ as the adjusting isomorphisms, we obtain a new normalized pseudo-identity $(F', \omega')$ which is isomorphic to $(F, \omega)$  as triangle functors,  and which satisfies $F'({\rm inc}_X)={\rm inc}_X$. We claim that $F'(i_{m, n})=i_{m, n}$. Indeed, this follows by applying $F'$ to
 $$i_{m, n}\circ {\rm inc}_{X_{m, n}}={\rm inc}_{X_{m-1, n}}.$$
 By the same argument, we infer that $F'$ acts on  $\{j_{m, n, a}, i'_{m, n, b}, \iota_{m, n, a, b}, \xi_{m, a, b}\}$ by the identity. Therefore,  $F'$ acts on all inclusions by the identity.

\emph{Step 2}\; By Step 1, up to adjustment,  we  may assume that the normalized pseudo-identity $(F, \omega)$ acts trivially on all the inclusions. For $m<n$, we consider the connection $c_{m, n-1, n-1}\colon X_{m, n-1}\rightarrow X_{n-1, n-1}$. We observe that
\begin{align}\label{equ:1}
F(c_{m, n-1, n-1})=c_{m, n-1, n-1}.
\end{align}
We denote by $\phi\colon P_0\rightarrow P_0$ the unique morphism with $\phi(e_0)=\alpha_{N-1}\cdots\alpha_1\alpha_0$. We have that $\Sigma^{1-n}(\phi)$ equals the composition $$X_{n-1,n-1}\xrightarrow{\rm inc} X_{m, n-1}\xrightarrow{c_{m, n-1, n-1}} X_{n-1, n-1}.$$
 By $F\Sigma^{1-n}(\phi)=\Sigma^{1-n}(\phi)$, we infer (\ref{equ:1}).

 We assume that $F(\pi_{m, n})=\lambda \pi_{m, n}$ for a nonzero scalar $\lambda$. We consider the following exact triangle in $\mathcal{T}$
\begin{align}\label{equ:2}
X_{n, n} \xrightarrow{\rm inc} X_{m, n}\xrightarrow{\pi_{m, n}}  X_{m, n-1} \xrightarrow{c_{m, n-1, n-1}} X_{n-1, n-1}.
\end{align}
Applying $(F, \omega)$ to this triangle and using the above observation, we have an exact triangle
\begin{align*}
X_{n, n} \xrightarrow{\rm inc} X_{m, n}\xrightarrow{\lambda\pi_{m, n}}  X_{m, n-1} \xrightarrow{c_{m, n-1, n-1}} X_{n-1, n-1}.
\end{align*}
Here, we use the fact that $\omega_{X_{n, n}}$ is the identity, since $(F, \omega)$ is normalized.  By Lemma \ref{lem:scalar}, we infer that $\lambda=1$ and thus $F(\pi_{m, n})=\pi_{m, n}$ for any $m<n$.

\emph{Step 3}\;  We claim that $F$ acts trivially on $\{\pi'_{m, n, a}, p_{m, n, b}, q_{m, n, a,b}, \zeta_{m, a,b}\}$. Consequently, $F$ acts trivially on all the projections.

To see $F(\pi'_{m, n, a})=\pi'_{m, n, a}$, we apply $(F, \omega)$ to the following exact triangle
\begin{align*}
X_{n, n} \xrightarrow{\rm inc} L_{m, n, a}\xrightarrow{\pi'_{m, n, a}}  L_{m, n-1, a} \xrightarrow{c_1} X_{n-1, n-1},
\end{align*}
where $c_1$ is the obvious connection. Similar to (\ref{equ:1}), we have $F(c_1)=c_1$. Applying Lemma \ref{lem:scalar} to the resulted exact triangle, we have the required identity.

 To see $F(p_{m, n, b})=p_{m, n, b}$, we apply $(F, \omega)$ to the following exact triangle
\begin{align*}
L_{n, n, b} \xrightarrow{\rm inc} R_{m, n, b}\xrightarrow{p_{m, n, b}}  X_{m, n-1} \xrightarrow{c_2} L_{n-1, n-1, b},
\end{align*}
where $c_2$ is the obvious connection. We have $F(c_2)=c_2$, which is similar to (\ref{equ:1}). Then the resulted exact triangle is as follows
\begin{align*}
L_{n, n, b} \xrightarrow{\rm inc} R_{m, n, b}\xrightarrow{F(p_{m, n, b})}  X_{m, n-1} \xrightarrow{c_2} L_{n-1, n-1, b},
\end{align*}
Here, we use the fact that $\omega_{L_{n, n, b}}$ is the identity. Recall that  $F(p_{m, n, b})$ differs from $p_{m, n, b}$ by a  nonzero scalar. Applying Lemma \ref{lem:scalar}, we are done.

Similarly, we apply $(F, \omega)$ to
\begin{align*}
L_{n, n, b} \xrightarrow{\rm inc} B_{m, n, a, b}\xrightarrow{q_{m, n, a, b}}  L_{m, n-1, a} \xrightarrow{c_{m, n-1,n-1, a, b}} L_{n-1, n-1, b},
\end{align*}
and
\begin{align*}
L_{m+1, m+1, b} \xrightarrow{\rm inc} Z_{m, a, b}\xrightarrow{\zeta_{m, a, b}}  L_{m, m, a} \xrightarrow{c_{m, m, m, a, b}} L_{m, m, b},
\end{align*}
to deduce $F(q_{m, n, a, b})=q_{m, n, a, b}$ and $F(\zeta_{m, a, b})=\zeta_{m, a, b}$, respectively.  This completes the proof of the claim.

We observe that $F(c_{l, m, n, a, b})=c_{l, m, n, a, b}$. The case $m=n$ is similar to (\ref{equ:1}). Assume that $m<n$. We apply $F$ to the composition
$$\Sigma^{-m}(\phi)=X_{m, m}\xrightarrow{\rm inc} L_{l, m, a}\xrightarrow{c_{l, m, n, a,b}} R_{m,n, b} \xrightarrow{\rm pr} X_{m, m}.$$
Here, $\phi$ denotes the unnamed arrow $P_0\rightarrow P_0$. We have $F\Sigma^{-m}(\phi)=\Sigma^{-m}(\phi)$.  Since $F$ acts trivially on all the inclusions and projections, we infer the required identity. Consequently, $F$ acts trivially on all the connections.

\emph{Step 4}\; We keep the assumptions in Step 2. We claim that $F(f)=f$ for any morphism $f$ of the mixed type. Consequently, $F$ acts trivially on the spanning set in Proposition \ref{prop:A}(1), as required.

We verify the claim for the $11$ classes of morphisms. The idea is to compose those morphisms with suitable inclusions,  and to use the fact that $F$ acts trivially on inclusions and connections.

Consider the morphism $L_{m, n, a}\rightarrow L_{m', n, b}$ of class ${\rm (mx.I)}$. We apply $F$ to the following commutative diagram
\begin{align}\label{equ:3}
\xymatrix{
X_{n, n} \ar@{=}[d] \ar[r]^-{\rm inc} & L_{m, n, a} \ar[d]^-{\rm (mx.I)}\\
X_{n, n} \ar[r]^-{\rm inc} & L_{m', n, b},
}
\end{align}
and deduce that $F$ acts trivially on the required morphism. The same argument works for ${\rm (mx.II)}$. For a morphism of class ${\rm (mx.III)}$, we apply $F$ to the following commutative diagram
\begin{align*}
\xymatrix{
L_{n, n, a} \ar[d]_-{c_{n, n, n', a, b}} \ar[r]^-{\rm inc} & R_{m, n, a} \ar[d]^-{\rm (mx.III)}\\
R_{n, n', b} \ar[r]^-{\rm inc} & R_{m, n', b},
}
\end{align*}
and use $F(c_{n, n, n', a, b})=c_{n, n, n', a, b}$. Here, we have to observe that the composition in the diagram is nonzero. The same argument works for all the remaining $8$ cases.
\end{proof}

\section{The case $r>1$}

In this section, we prove Theorem \ref{thm:A} for the case $r>1$ and $N>r$. The idea of the proof is the same as in Section 4, so we omit some details.

Let $A=A(r, N)$, $\mathcal{A}=A\mbox{-proj}$ and $\mathcal{T}=\mathbf{K}^b(\mathcal{A})$.  We use Convention ($\dag$) in Section 3. In particular, the lower indices $s$ of $P_s$ are taken in $\mathbb{Z}/r\mathbb{Z}$.  We consider the following objects in $\mathcal{T}$:
\begin{align*}
X_{s, m, n}& =\cdots \rightarrow 0\rightarrow P_s\rightarrow P_{s-1}\rightarrow \cdots \rightarrow P_{s-n+m+1}\rightarrow P_{s-n+m}\rightarrow 0\rightarrow \cdots\\
L_{m, n, a} &= \cdots \rightarrow 0\rightarrow Q_a\rightarrow P_{r-1}\rightarrow P_{r-2}\rightarrow \cdots \rightarrow P_{r-n+m+1}\rightarrow P_{r-n+m} \rightarrow 0\rightarrow \cdots\\
R_{m, n, b} &= \cdots \rightarrow 0\rightarrow P_{n-m-1}\rightarrow P_{n-m-2}\rightarrow \cdots \rightarrow P_1\rightarrow P_0 \rightarrow Q_b \rightarrow 0\rightarrow \cdots\\
B_{m, n, a, b} &=\cdots \rightarrow 0\rightarrow Q_a\rightarrow P_{r-1}\rightarrow P_{r-2}\rightarrow \cdots \rightarrow P_1\rightarrow P_0 \rightarrow Q_b \rightarrow 0 \rightarrow \cdots\\
Z_{m, a, b} &=\cdots \rightarrow 0\rightarrow Q_a\rightarrow Q_b \rightarrow 0\rightarrow \cdots.
\end{align*}
The supports of $X_{s, m, n}$, $L_{m, n, a}$ and $R_{m, n, b}$ are $[m, n]$ for $m\leq n$, where $s$ lies in $\mathbb{Z}/{r\mathbb{Z}}$, and $a, b$ run from $r$ to $N-1$. We observe that $X_{s, m, m}=\Sigma^{-m}(P_s)$ and $L_{m, m,a}=R_{m, m, a}=\Sigma^{-m}(Q_a)$. For $m< n-r$ with $r|(n-m-1)$ and $r\leq a, b< N$, the complex $B_{m, n, a, b} $ is defined, whose support is $[m, n]$. The complex $Z_{m, a, b}$ is supported in $[m, m+1]$, where $r\leq b<a< N$.

The following result is analogous to Lemma \ref{lem:ind}.

\begin{lem}
Keep the above notation. Then
$$\Lambda=\{X_{s, m, n}, L_{m, n, a}, R_{m, n, b}, B_{m, n, a, b}, Z_{m, a, b}\}$$
is a complete set of representatives  of isoclasses of indecomposable objects in $\mathcal{T}$. Moreover, for each object $Y$ in $\Lambda$, we have ${\rm End}_\mathcal{T}(Y)=k{\rm Id}_Y$. \hfill $\square$
\end{lem}

There are four types of morphisms between objects in $\Lambda$: \emph{inclusions}, \emph{projections}, \emph{connections} and morphisms of the \emph{mixed} type; compare \cite{ALP} and Section 4.

The first type is induced by the obvious inclusion:
$i_{s, m, n}\colon X_{s, m, n}\rightarrow X_{s+1, m-1, n}$, $j_{m, n,a}\colon X_{r-1, m, n}\rightarrow L_{m-1, n, a}$, $i'_{m, n, b}\colon R_{m, n, b}\rightarrow R_{m-1, n, b}$ for $s\in \mathbb{Z}/{r\mathbb{Z}}$, $m\leq n$ and $r \leq a, b< N$; $\iota_{m, n, a, b}\colon R_{m, n, b}\rightarrow B_{m-1, n, a, b}$ for $m\leq n-r$ with $r|(n-m)$ and $r\leq a, b< N$; $\xi_{m, a, b}\colon L_{m, m, b}\rightarrow Z_{m-1, a, b}$ for each integer $m$ and $r\leq b<a<N$. We will call any composition of these morphisms an \emph{inclusion}, which will be denoted by ``inc".

The second type is induced by the obvious projection: $\pi_{s, m, n}\colon X_{s, m, n}\rightarrow X_{s, m, n-1}$, $\pi'_{m, n, a}\colon L_{m, n, a}\rightarrow L_{m, n-1, a}$ and $p_{m, n, b}\colon R_{m, n, b}\rightarrow X_{n-m-1, m, n-1}$ for $s\in \mathbb{Z}/{r\mathbb{Z}}$, $m<n$ and $r\leq a, b < N$; $q_{m, n, a, b}\colon B_{m, n, a, b}\rightarrow L_{m, n-1, a}$ for $m< n-r$ with $r|(m-n-1)$  and $r\leq a, b < N$; $\zeta_{m, a, b}\colon Z_{m, a, b}\rightarrow L_{m, m, a}$ for each integer $m$ and $r\leq b<a< N$. We will call any composition of these morphisms a \emph{projection}, which will be denoted by ``pr".

 For the third type, we denote by $c_{l, m, n, a, b}\colon L_{l, m, a}\rightarrow R_{m, n, b}$  the following morphism
\[\xymatrix{
0\ar[r] & Q_a\ar[r] & P_{r-1}\ar[r] & \cdots \ar[r] & P_{r-m+l}\ar[d]\ar[r]  & 0\\
&             &              &               0\ar[r] & P_{n-m-1} \ar[r] &  \cdots \ar[r] &  P_0\ar[r] & Q_b\ar[r] &  0.
}\]
for $l\leq m \leq n$  satisfying  $r|(n-l)$ and $r\leq a, b< N$. Here, the unique nonzero vertical map is the unnamed arrow $P_{r-m+l}\rightarrow P_{n-m-1}$, where we use the fact $r|(n-l)$. If $l=m<n$, the vertical map is given by  $Q_a\rightarrow P_{r-1}$; if $l<m=n$, it is given by $P_0\rightarrow Q_b$. If $l=m=n$, we have to assume that $b<a$, in which case we have $c_{l, l, l, a, b}=\Sigma^{-l}(\phi)$, where $\phi\colon Q_a\rightarrow Q_b$ is the unnamed arrow. We will call any morphism of the form $${\rm pr}\circ c_{l, m, n, a, b}\circ {\rm inc}$$
 a \emph{connection}.

The fourth type, called the \emph{mixed} type, is  divided into $11$ classes of morphisms. They are listed as follows.
\[
\xymatrix@!=8pt{
L_{m, n, a}\ar[d]_{{\rm (mx.I)}}\colon &  0\ar[r] & Q_a \ar[r] & P_{r-1}\ar[r] & \cdots \ar[r] & P_0\ar[d] \ar[r] &P_{r-1} \ar@{=}[d]\ar[r] & \cdots \ar[r] & P_{m-n} \ar@{=}[d] \ar[r]  & 0 \\
L_{m', n, b}\colon & &  & & 0\ar[r] & Q_b \ar[r] & P_{r-1}\ar[r] & \cdots \ar[r] & P_{m'-n}\ar[r] & 0,
}\]
for $m<m'<n$ with $r|(m'-m)$  and $r\leq a, b < N$;
\[
\xymatrix@!=10pt{
L_{m, n, a}\ar[d]_{{\rm (mx.II)}}\colon &  0\ar[r] & Q_a \ar[d]\ar[r] & P_{r-1}\ar@{=}[d]\ar[r] & P_{r-2}\ar@{=}[d]\ar[r] &  \cdots \ar[r]  & P_{m-n} \ar@{=}[d] \ar[r]  & 0 \\
L_{m, n, b}\colon & 0\ar[r] & Q_b \ar[r] & P_{r-1}\ar[r] & P_{r-2}\ar[r] & \cdots \ar[r]  & P_{m-n} \ar[r]  & 0,
}\]
for $m<n$ and $r\leq b<a < N$;
\[
\xymatrix @C=10pt{
R_{m, n, a}\ar[d]_{{\rm (mx.III)}}\colon &  0\ar[r] & P_{n-m-1} \ar@{=}[d]\ar[r] &  \cdots \ar[r] & P_1 \ar@{=}[d] \ar[r]& P_0\ar@{=}[d]\ar[r] & Q_a\ar[d] \ar[r]  & 0 \\
R_{m, n', b}\colon & 0\ar[r] & P_{n'-m-1} \ar[r] & \cdots \ar[r] & P_1 \ar[r] & P_0\ar[r] & P_{r-1}\ar[r] & \cdots \ar[r]  & P_0 \ar[r]  & Q_b\ar[r] &  0,
}\]
for $m<n<n'$  with $r|(n'-n)$  and $r\leq a, b < N$;
\[
\xymatrix @C=10pt{
R_{m, n, a}\ar[d]_{{\rm (mx.IV)}}\colon &  0\ar[r] & P_{n-m-1} \ar@{=}[d]\ar[r] &  \cdots \ar[r]& P_1 \ar@{=}[d]\ar[r] & P_0\ar@{=}[d]\ar[r] & Q_a\ar[d] \ar[r]  & 0 \\
R_{m, n, b}\colon & 0\ar[r] & P_{n-m-1} \ar[r] & \cdots \ar[r]   & P_1\ar[r] &  P_0 \ar[r]  & Q_b\ar[r] &  0,
}\]
for $m<n$ and $r\leq b<a < N$;
\[
\xymatrix @C=10pt{
B_{m, n, a, b} \ar[d]_{{\rm (mx.V)}} \colon  & 0 \ar[r] & Q_a  \ar[r] &  \cdots \ar[r] & P_0\ar[d]\ar[r] & P_{r-1}\ar@{=}[d]\ar[r] & \cdots \ar[r]&  P_0 \ar@{=}[d]\ar[r] & Q_b \ar[d]\ar[r] & 0 \\
B_{m', n', a', b'} \colon  &  & &   0 \ar[r] & Q_{a'} \ar[r] & P_{r-1} \ar[r] & \cdots \ar[r] &  P_0\ar[r] & P_{r-1} \ar[r] & \cdots \ar[r] &  Q_{b'} \ar[r] & 0,
}\]
for $m<m'<n<n'$  with $r|(m'-m)$,  $r|(n'-n)$ and $r\leq a, b, a', b'< N$;
\[
\xymatrix @C=10pt{
B_{m, n, a, b} \ar[d]_{{\rm (mx.VI)}} \colon & 0 \ar[r] & Q_a\ar[d] \ar[r] & P_{r-1}\ar@{=}[d] \ar[r] & \cdots \ar[r] & P_0\ar[r] \ar@{=}[d] &Q_b \ar[d]\ar[r] & 0 \\
B_{m, n', a', b'} \colon  &  0 \ar[r] & Q_{a'} \ar[r] & P_{r-1} \ar[r] & \cdots \ar[r] & P_0 \ar[r] & P_{r-1} \ar[r] &  \cdots \ar[r] & Q_{b'} \ar[r] & 0,
}\]
for $m+r< n<n'$ with $r|(n-m-1)$,  $r|(n'-n)$  and $r\leq a, b, a', b'< N$ satisfying  $a'\leq a$,  where for the case $a=a'$ the unnamed arrow $Q_a\rightarrow Q_{a'}$ is the identity map; dually, we have
\[
\xymatrix @C=10pt{
B_{m, n, a, b} \ar[d]_{{\rm (mx.VII)}} \colon & 0 \ar[r] & Q_a \ar[r]  & \cdots \ar[r] & P_0\ar[d]\ar[r] & P_{r-1}\ar@{=}[d]\ar[r] & \cdots \ar[r] & P_0 \ar@{=}[d] \ar[r] &  Q_b \ar[d]\ar[r] & 0 \\
B_{m', n, a', b'} \colon  & & & 0 \ar[r] & Q_{a'} \ar[r] & P_{r-1} \ar[r] & \cdots \ar[r] & P_0\ar[r] & Q_{b'}\ar[r] & 0,
}\]
for $m<m'< n-r$ with $r|(m'-m)$, $r|(n-m-1)$ and $r\leq a, b, a', b'< N$ satisfying $b'\leq b$, where for the case $b=b'$ the unnamed arrow $Q_b\rightarrow Q_{b'}$ is the identity map;
\[
\xymatrix @C=10pt{
B_{m, n, a, b} \ar[d]_{{\rm (mx.VIII)}} \colon & 0 \ar[r] & Q_a\ar[d]  \ar[r] & P_{r-1}\ar@{=}[d]\ar[r] &\cdots \ar[r] & P_0\ar@{=}[d]\ar[r]  & Q_b \ar[d]\ar[r] & 0 \\
B_{m, n, a', b'} \colon  &  0 \ar[r] & Q_{a'} \ar[r] & P_{r-1} \ar[r] & \cdots \ar[r] & P_0\ar[r] &  Q_{b'}\ar[r] & 0,
}\]
for $m< n-r$ with $r|(n-m-1)$, $a'\leq a$ and $b'\leq b$;
\[
\xymatrix @C=10pt{
B_{m, n, a, b} \ar[d]_{{\rm (mx.IX)}} \colon & 0 \ar[r] & Q_a  \ar[r] & P_{r-1}\ar[r] &\cdots \ar[r] & P_0\ar[d]\ar[r]  & Q_b \ar[d]\ar[r] & 0 \\
Z_{n-1, a', b'} \colon  &   & & & 0 \ar[r] & Q_{a'} \ar[r] &   Q_{b'}\ar[r] & 0,
}\]
for $m<n-r$ with $r|(n-m-1)$ and $r\leq b'\leq b<a'< N$;
\[
\xymatrix @C=10pt{
Z_{m, a, b} \ar[d]_{{\rm (mx.X)}}\colon &  0 \ar[r] & Q_a \ar[d] \ar[r] & Q_b \ar[d] \ar[r] & 0\\
B_{m, n, a', b'}\colon & 0 \ar[r] & Q_{a'} \ar[r] & P_{r-1} \ar[r] & \cdots \ar[r] & P_0 \ar[r] & Q_{b'} \ar[r] & 0
}\]
for $m\leq n-r$ with $r|(n-m-1)$  and $r\leq b<a'\leq a<N$;
\[
\xymatrix @!=8pt{
Z_{m, a, b} \ar[d]_{{\rm (mx.XI)}}\colon &  0 \ar[r] & Q_a \ar[d] \ar[r] & Q_b \ar[d] \ar[r] & 0\\
Z_{m,  a', b'}\colon & 0 \ar[r] & Q_{a'} \ar[r]  & Q_{b'} \ar[r] & 0
}\]
for each integer $m$ and $r\leq b'\leq b<a'\leq a< N$.

The following result can be proved by the same argument as in Proposition \ref{prop:A}. The second statement is contained in \cite[Theorem 6.1]{BPP}.

\begin{prop}\label{prop:B}
Keep the notation as above. Then the following statements hold.
\begin{enumerate}
\item The above morphisms
\begin{align*}
&\{i_{s, m, n}, j_{m, n, a}, i'_{m, n, b}, \iota_{m, n, a, b}, \xi_{m, a, b}\}\cup \{\pi_{s, m, n},\pi'_{m, n, a}, p_{m, n, b}, q_{m, n, a,b}, \zeta_{m, a,b} \}\\
&\cup \{c_{l, m, n, a, b}\}\cup \{{\rm (mx.I)-(mx.XI)}\}
\end{align*}
form a spanning set of  $\mathcal{T}$.
\item For two objects $X, Y\in \Lambda$, we have ${\rm dim} \; {\rm Hom}_\mathcal{T}(X, Y)\leq 1$. \hfill $\square$
\end{enumerate}
\end{prop}

We now sketch a proof of Theorem \ref{thm:A} for the case $r>1$, which  is almost identical to the one in Section 4. We only indicate  necessary changes.

\begin{proof}
Let $(F, \omega)$ be a pseudo-identity on $\mathcal{T}$. By Lemma \ref{lem:normal}, we assume that $(F, \omega)$ is normalized.

It suffices to adjust $(F, \omega)$ by nonzero scalars such that the adjusted pseudo-identity  $(\bar{F}, \bar{\omega})$ acts trivially on the spanning set in   Proposition \ref{prop:B}(1). We observe that Proposition \ref{prop:A2} applies to $\mathcal{T}$, where the subset $S$ of $\Lambda={\rm ind}\mathcal{T}$ is empty. Thanks to Proposition \ref{prop:A2}(1), the homomorphism (\ref{equ:res}) is injective. By Lemma \ref{lem:span} and Proposition \ref{prop:A2}(2), we infer that $ (\bar{F}, \bar{\omega})$ is isomorphic to the identity functor.  Therefore, the category $\mathcal{A}$ is strongly $\mathbf{K}$-standard.

For the required adjustment, we proceed as in Section 4. The first step is almost the same, just replacing $X_{m, n}$ by $X_{s, m, n}$. Consequently, the adjusted pseudo-identity  is still normalized and acts trivially on all the inclusions.

For the second step, (\ref{equ:1}) is replaced by $F(c_{s, m, n-1, n-1})=c_{s, m, n-1, n-1}$. Here,
$$c_{s, m, n-1, n-1}\colon X_{s, m, n-1}\longrightarrow X_{s-n+m, n-1, n-1}$$
is the obvious connection, whose $(n-1)$-th degree is given by the unnamed arrow $P_{s-n+m+1}\rightarrow P_{s-n+m}$.  Then the triangle (\ref{equ:2}) is  replaced by the following one
\begin{align*}
X_{s-n+m, n, n} \xrightarrow{\rm inc} X_{s, m, n}\xrightarrow{\pi_{s, m, n}}  X_{s, m, n-1} \xrightarrow{c_{s, m, n-1, n-1}} X_{s-n+m, n-1, n-1}.
\end{align*}
Applying $(F, \omega)$ to it and using Lemma \ref{lem:scalar}, we infer that $F(\pi_{s, m,n})=\pi_{s, m, n}$. Here, we use the fact that $\omega_{X_{s-n+m, n, n}}$ is the identity, since the given pseudo-identity is normalized. The remaining part of Step 2 carries over as in Section 4. The third step is the same as  Step 3 in Section 4.

For the last step, we just repeat Step 4 in Section 4. For example, the commutative diagram (\ref{equ:3}) is replaced by
\[
\xymatrix{
X_{m-n, n, n} \ar@{=}[d] \ar[r]^-{\rm inc} & L_{m, n, a} \ar[d]^-{\rm (mx.I)}\\
X_{m'-n, n, n} \ar[r]^-{\rm inc} & L_{m', n, b}.
}\]
Here,  we recall that $r|(m'-m)$ in defining ${\rm (mx.I)}$, and thus $X_{m-n, n, n}=X_{m'-n, n, n}$. We omit the details.
\end{proof}

\vskip 10pt

\noindent {\bf Acknowledgements}\quad  The authors are supported by National Natural Science Foundation of China (No.s 11522113,  11671245 and 11601098), and  Natural Science Foundation of Guizhou Province (QSF[2016]1038).

\bibliography{}

\vskip 10pt

 {\footnotesize \noindent Xiao-Wu Chen\\
 Key Laboratory of Wu Wen-Tsun Mathematics, Chinese Academy of Sciences\\
School of Mathematical Sciences, University of Science and Technology of China\\
No. 96 Jinzhai Road, Hefei, 230026, Anhui, P.R. China.\\
URL: http://home.ustc.edu.cn/$^\sim$xwchen}

\vskip 5pt

 {\footnotesize \noindent Chao Zhang\\
 Department of Mathematics, Guizhou University, Guiyang 550025, Guizhou,  P.R. China.}

\end{document}